\documentclass{amsart}
\usepackage{amssymb}
\usepackage[alphabetic]{amsrefs}
\usepackage{graphicx}
\usepackage{enumerate}

\renewcommand{\div}{\operatorname{div}}
\newcommand{\loc}{\mathrm{loc}}
\newcommand{\R}{\mathbb{R}}
\newcommand{\N}{\mathbb{N}}
\newcommand{\K}{\mathcal{K}}
\renewcommand{\d}{\,\textup{d}}

\newcommand{\dist}{\operatorname{dist}}

\renewcommand{\phi}{\varphi} 
\renewcommand{\epsilon}{\varepsilon} 
\renewcommand{\Re}{\operatorname{Re}}

\newtheorem{theorem}{Theorem}[section]
\newtheorem{lemma}[theorem]{Lemma}
\newtheorem{proposition}[theorem]{Proposition}

\theoremstyle{remark}
\newtheorem{remark}[theorem]{Remark}
\theoremstyle{definition}

\numberwithin{equation}{section}

\title[Contact of a thin free boundary with a fixed one]{Contact of a
  thin free boundary\\with a fixed one in the Signorini problem}

\author{Norayr Matevosyan}
\address{Department of Mathematics, University of Texas at Austin,
  Austin, TX 78712, USA}
\email{nmatevosyan@math.utexas.edu}
\thanks{N.M. was supported by Award No.~KUK-I1-007-43, made by King Abdullah University of Science and Technology (KAUST)}

\author{Arshak Petrosyan}
\address{Department of Mathematics, Purdue University, West Lafayette,
IN 47907, USA}
\email{arshak@math.purdue.edu}
\thanks{A.P. was supported in part by  NSF grant DMS-1101139.}


\begin{document}
\begin{abstract}
  We study the Signorini problem near a fixed boundary, where the
  solution is ``clamped down'' or ``glued.''  We show that in general
  the solutions are at least $C^{1/2}$ regular and that this regularity is
  sharp. We prove that near the actual points of contact of the free
  boundary with the fixed one the blowup solutions must have
  homogeneity $\kappa\geq 3/2$, while at the non-contact points the
  homogeneity must take one of the values: $1/2$, $3/2$, \ldots,
  $m-1/2$, \ldots.
\end{abstract}

\maketitle

\section{Introduction and Main Results}
\label{sec:intr-main-result}

\subsection{The Signorini problem}
\label{sub:the_problem}

The purpose of this paper is to study the behavior of the thin free
boundary as it approaches to the fixed boundary in the so-called
(scalar) \emph{Signorini problem} (also know as the \emph{thin
  obstacle problem}).

The Signorini problem consists in minimizing the Dirichlet energy
functional
\begin{equation}
  J(v):=\int_{B_1^+} |\nabla v|^2
  \label{eq:energy}
\end{equation}
on a closed convex set
\begin{equation}
  \label{eq:K}
  \K=\K(g):=\{v\in W^{1,2}(B_1^+): v=g\text{ on } (\partial B_1)^+,\ 
  v\geq 0\text{ on } B_1'\},
\end{equation}
for a given function $g\in L^2((\partial B_1)^+)$.  Here and
everywhere in the paper we use the following notations:
\begin{gather*}
  B_r(x):= \{y\in \R^n :|x-y|<r\},\quad B_r:= B_r(0),\\
  E^+:= E \cap \{x_n>0 \},\quad E':= E \cap \{ x_n=0 \},
\end{gather*}
for a subset $E\subset \R^n$. We assume $n\geq 2$.
Using direct methods of calculus of variation one
can verify that a minimizer $u\in \K$ exists and satisfies the
following variational inequality:
\begin{equation} \label{eq:var-ineq} \int_{B_1^+} \nabla u \nabla (v-u)
  \geq 0 \quad\text{for any }v\in \K.
\end{equation}
The problem above goes back to the foundational paper \cite{LS} on
variational inequalities. It is has been known for quite some time
that the minimizers are in the class $C^{1,\alpha}(B_1^+\cup B_1')$
for some $\alpha>0$ (see \cite{Ca1} and also \cite{U1}) and even
$C^{1,1/2}(B_1^+\cup B_1')$ in the dimension $n=2$, see
\cite{Ri}. Besides, the minimizers satisfy
\begin{align*}
  \Delta u=0&\quad\text{in }B_1^+\\
  u\geq 0,\quad -\partial_{x_n} u\geq 0,\quad
  u\partial_{x_n}u=0&\quad\text{on }B_1'.
\end{align*}
The latter are known as the \emph{Signorini} or \emph{complementarity
  boundary conditions}.

The problem features the following apriori unknown subsets of $B_1'$:
\begin{align*}
  &\Lambda(u):=\{x\in B_1': u=0\}&& \text{\emph{the coincidence set}}\\
  & \Omega(u):=\{x\in B_1': u>0\}&& \text{\emph{the non-coincidence set}}\\
  & \Gamma(u):=\partial_{B_1'}\Omega(u) &&\text{\emph{the free
      boundary.}}
\end{align*}
The study of the geometric and analytic properties of the free
boundary is one of the objectives of the Signorini problem. Sometimes
it is said that the free boundary $\Gamma(u)\subset B_1'$ is
\emph{thin}, to indicate that it is (expected to be) of dimension
$(n-2)$.

Recent years have seen some interesting new developments in the
problem, starting with the proof in \cite{AC} that the minimizers $u$
are in the class $C^{1,1/2}(B_1^+\cup B_1')$, in any dimension $n\geq
2$, which is the optimal regularity. This opened up the possibility of
studying the free boundary $\Gamma(u)$, which has been done in
\cites{ACS,CSS,GP}, see also \cite[Chapter~9]{PSU}. An effective tool
in the study of the free boundary is \emph{Almgren's frequency
  formula}
$$
N^x(r,u):=\frac{r\int_{B_r^+(x)} |\nabla u|^2}{\int_{(\partial B_r)^+}
  u^2}.
$$
It originated in the work of Almgren on multi-valued harmonic
functions \cite{Alm} and has an important property of being monotone
in $r$, even for solutions of the Signorini problem. One then
classifies the free boundary points according to the value
$$
\kappa:=N^x(0+,u).
$$
It is known that $\kappa\geq 3/2$ for $x\in\Gamma(u)$ in the Signorini problem and more precisely $\kappa=3/2$ or
$\kappa\geq 2$ \cite{ACS}. This results in a decomposition
$$
\Gamma(u)=\Gamma_{3/2}(u)\cup\bigcup_{\kappa\geq2}\Gamma_\kappa(u),\quad
\text{where }\Gamma_\kappa(u):=\{x\in\Gamma(u):N^x(0+,u)=\kappa\}.
$$
The set $\Gamma_{3/2}(u)$ is known as the \emph{regular set}. It has been
recently shown that $\Gamma_{3/2}(u)$ is real analytic \cite{KPS} by
using a partial hodograph-Legendre transform from $C^{1,\alpha}$
regularity proved in \cite{ACS}. See also \cite{DS1} for a different
proof of $C^\infty$ regularity, based on a generalization of the
boundary Harnack principle. The only other free boundary points
studied in the literature 
are the ones in $\Gamma_{2m}(u)$, $m\in\N$ which correspond to
the points where the coincidence set $\Lambda(u)$ has a zero $H^{n-1}$
density, see \cite{GP}. Such points are known as \emph{singular
  points}. It was proved in \cite{GP} that $\Gamma_{2m}(u)$ is
contained in a countable union of $C^1$ manifolds.

An interesting question is finding all possible values for
$\kappa=N^x(0+,u)$. In dimension $n=2$ the answer to that question is
known (proof is a simple exercise): $\kappa$ must be one of the
following values:
$$
3/2,\ 2,\ 7/2,\ 4,\ \ldots,\ 2m-1/2,\ 2m,\ \ldots.
$$
However, this is still an open problem in dimensions $n\geq 3$.

\subsection{Contact of the free and fixed boundaries} The objective in
this paper is the study of the behavior of the free boundary
$\Gamma(u)$ in the Signorini problem as it approaches a set where $u$
is forced to be zero. More precisely, consider a closed subset $\K_0$
of the set $\K$ in \eqref{eq:K}, defined by
\begin{equation}
  \label{eq:K0}
  \K_0=\K_0(g):=\{v\in \K(g): v=0\text{ on } B_1'\cap\{x_1\leq 0\}\}
\end{equation}
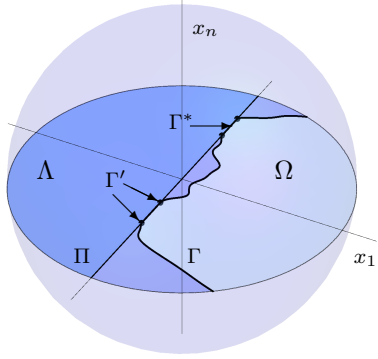
\begin{figure}[t]
  \centering
  \begin{picture}(150,150)
    \put(0,0){\includegraphics[height=150pt]{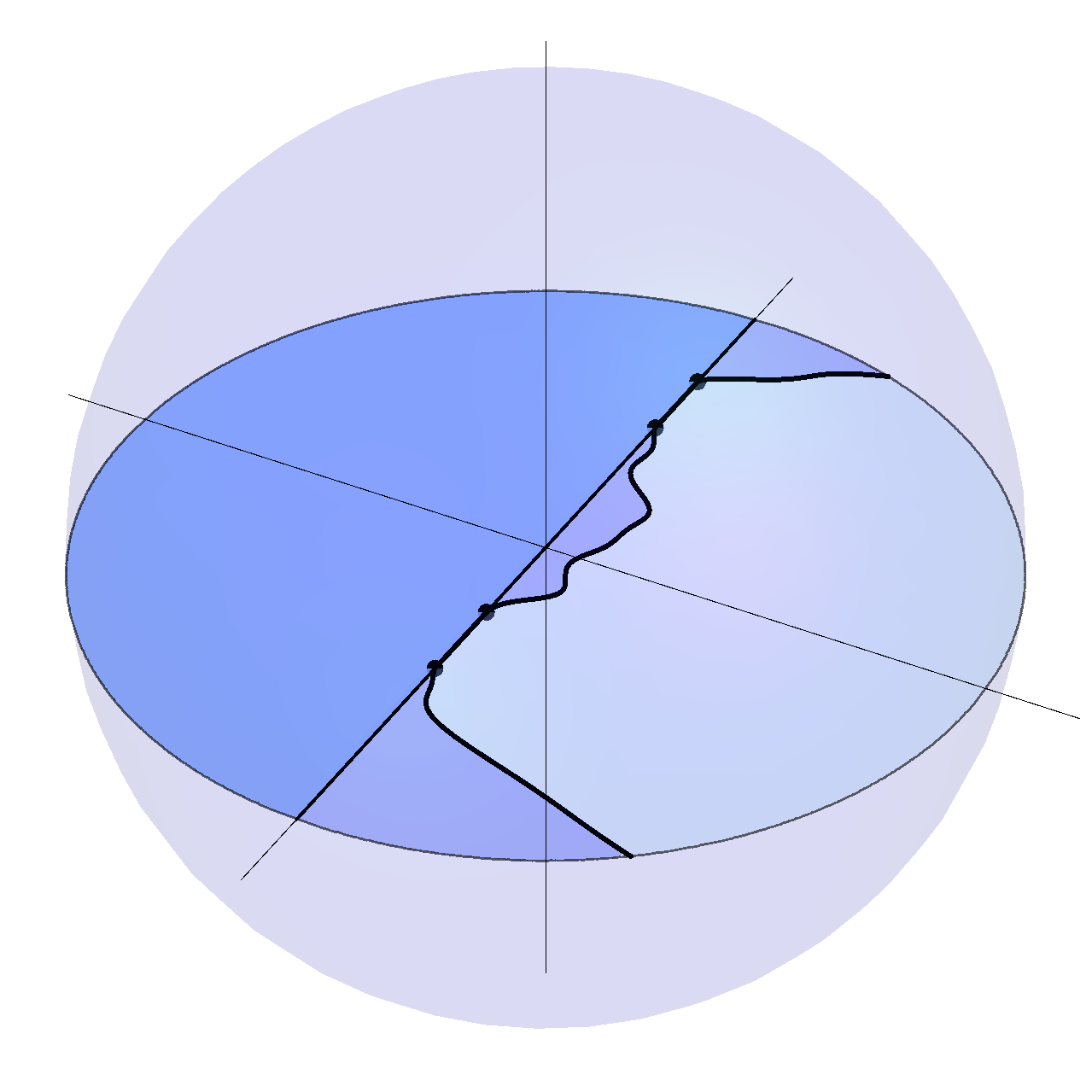}}
    \put(79,130){\footnotesize $x_n$} \put(140,43){\footnotesize
      $x_1$} \put(34,43){\footnotesize $\Pi$}
    \put(46,72){\footnotesize $\Gamma'$} \put(70,94){\footnotesize
      $\Gamma^*$} \put(77,43){\footnotesize $\Gamma$}
    \put(20,75){$\Lambda$} \put(110,75){$\Omega$}
    \put(49,69){\vector(1,-1){10}} \put(53,73){\vector(2,-1){13}}
    \put(78,95){\vector(1,0){16}}
  \end{picture}
  \caption{Free boundary $\Gamma $ near the contact points $\Gamma'$
    with the fixed boundary $\Pi $ considered in the hyperplane
    $\{x_n=0\}$.}
  \label{fig:contact}
\end{figure}%
and minimize the Dirichlet energy $J$ in \eqref{eq:energy} over
$\K_0$. That is, compared to the Signorini problem, we have an
additional constraint that the functions must vanish on
$B_1'\cap\{x_1\leq 0\}$.  If we think of the solution of the Signorini
problem as an elastic membrane that is forced to stay above zero in
$B_1'$, the new constraint in $\K_0$ can be thought of as ``clamping
down'' or ``gluing'' the membrane on $B_1'\cap\{x_1\leq 0\}$. The
boundary of the latter set in $B_1'$ is
$$
\Pi:=\{x_1=0,x_n=0\},
$$
which we call the \emph{fixed boundary}. Note that the coincidence set
$\Lambda(u)$ will contain now $B_1'\cap\{x_1\leq 0\}$ and the truly
free part of $\Gamma(u)$ is $\Gamma(u)\cap\{x_1>0\}$. The points in
$$
\Gamma'(u):=\overline{\Gamma(u)\cap\{x_1>0\}}\cap \Pi
$$
are categorized as \emph{contact points}, and the ones in
$$
\Gamma^*(u):=(\Gamma(u)\cap\Pi)\setminus \Gamma'(u)
$$
are \emph{non-contact points}, see Fig.~\ref{fig:contact}. We note that the minimizers in $\K_0$
still solve the Signorini problem in small halfballs $B_r^+(x_0)$ with
$x_0\in B_1'\cap\{x_1>0\}$ and therefore we will have that $u\in
C_{\loc}^{1,1/2}(B_1^+\cup (B_1'\cap\{x_1>0\}))$ and that it satisfies
\begin{align*}
  \Delta u=0&\quad\text{in }B_1^+\\
  u=0&\quad\text{on }B_1'\cap \{x_1\leq 0\}\\
  u\geq 0,\quad -\partial_{x_n} u\geq 0,\quad
  u\partial_{x_n}u=0&\quad\text{on }B_1'\cap\{x_1>0\} .
\end{align*}
There are many papers in the literature dealing with the contact of
the free and fixed boundaries in various free boundary problems. The
case of the classical obstacle problem, for instance, was studied by
\cites{U2,AU}. We also refer to \cite[Chapter~8]{PSU} and references
therein for some of these results, including also extensions to other
obstacle-type problems.

In contrast to the case of the classical obstacle problem, where the
presence of the fixed boundary actually helps -- for instance, to
avoid a geometric ``thickness'' condition on coincidence set needed
for the regularity of the free boundary -- in the Signorini problem
the presence of the fixed boundary introduces a serious handicap.
Indeed, as we have mentioned earlier, the optimal regularity of the
Signorini problem is $C^{1,1/2}$. This regularity is exhibited by
the following explicit solution:
\begin{equation} \label{eq:reZ3/2} \hat
  u_{3/2}(x):=\Re(x_1+i|x_n|)^{3/2}.
\end{equation}
On the other hand, it is easy to see that
\begin{equation} \label{eq:reZ} \hat
  u_{1/2}(x):=\Re(x_1+i|x_n|)^{1/2}
\end{equation}
is a minimizer of $J$ over $\K_0$ (simply because it is harmonic in
$B_1\setminus (B_1'\cap\{x_1\leq 0\})$), thus limiting the generally
expected regularity of minimizers of $J$ to at most
$C^{1/2}$. (See Fig.~\ref{fig:REZ} for the illustration of these solutions.)

This lower regularity of minimizers undercuts many techniques used for
the Signorini problem, calling for caution even when dealing with the
first derivatives of the solution. Luckily, however, one of the most
important tools in our analysis, Almgren's frequency formula, still
works: one of the steps in the proof is based on a Rellich-type
identity, which in our case becomes an inequality in the correct
direction and allows the proof to go through.
\begin{figure}[t]
  \begin{picture}(320,140)(0,0)
    \put(0,20){\includegraphics[width=150pt]{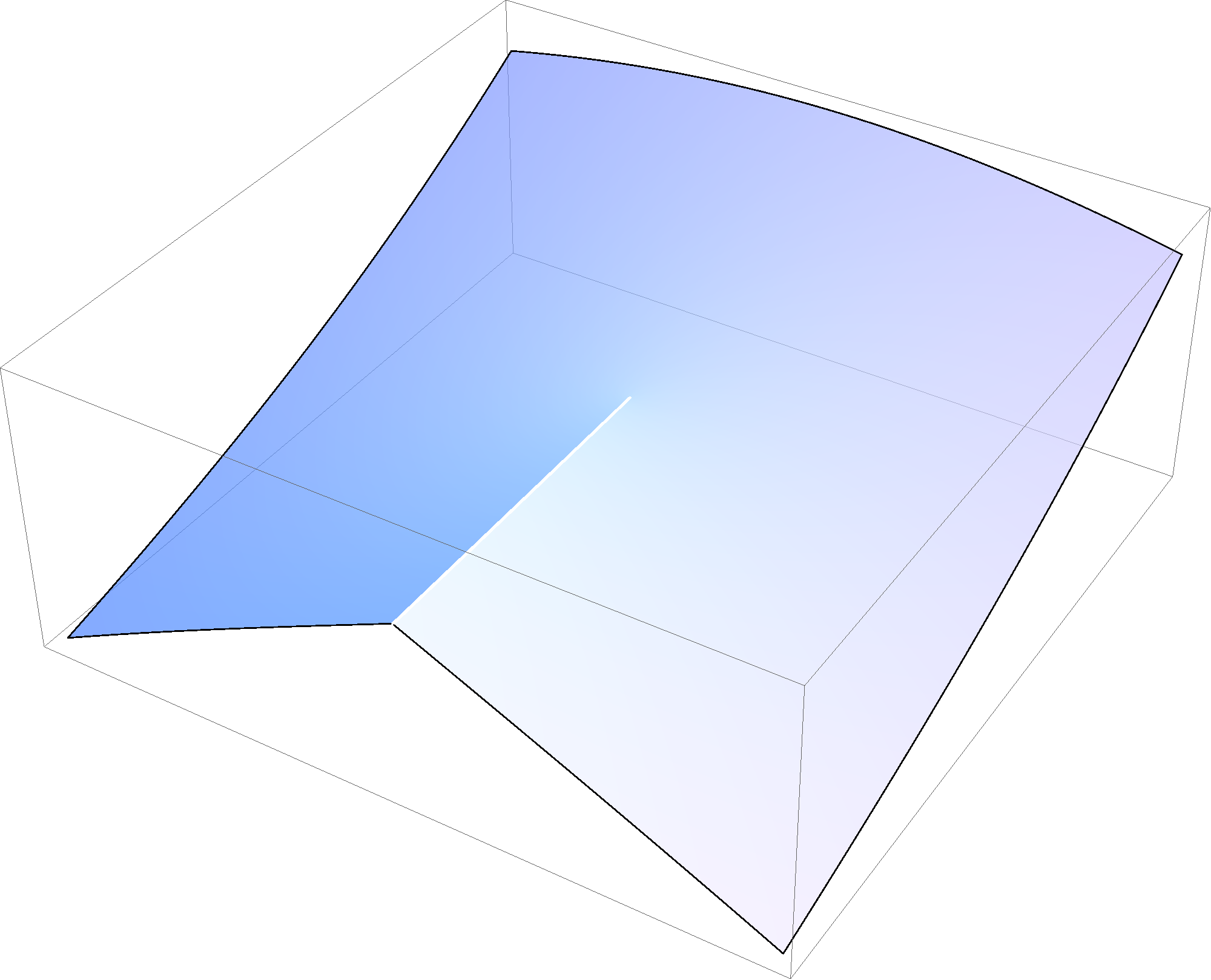}}
    \put(160,20){\includegraphics[width=150pt]{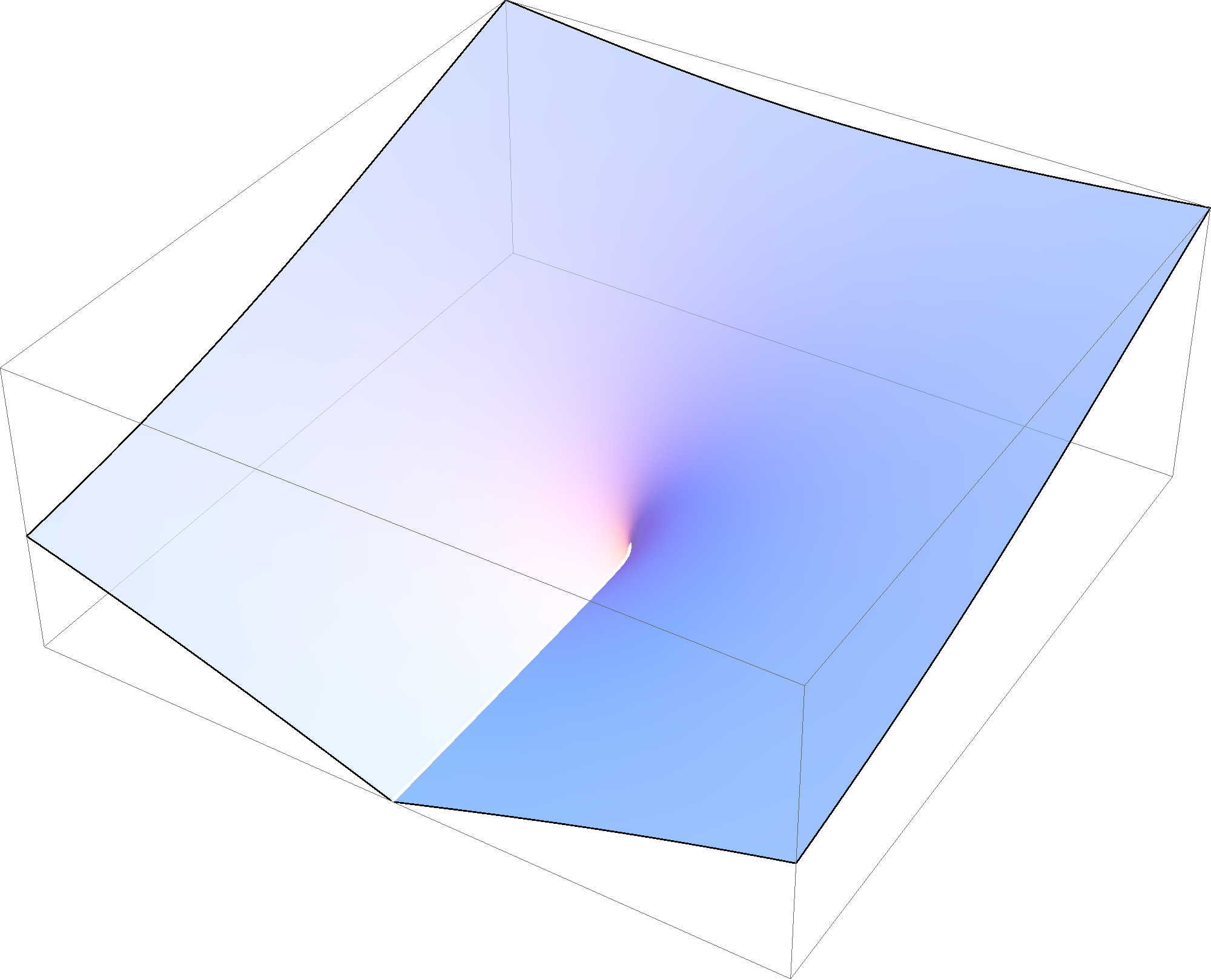}}

    \put(160,0){$\hat u_{1/2}(x)=\Re(x_1+i|x_n|)^{1/2}$}
    \put(0,0){$\hat u_{3/2}(x)=\Re(x_1+i|x_n|)^{3/2}$} 
  \end{picture}
  \caption{Examples of solutions limiting the optimal regularity:
    $\hat u_{3/2}(x)$ is an explicit solution of the Signorini problem
    and $\hat u_{1/2}(x) $ is a minimizer over $\K_0$ with worst
    possible regularity.}
  \label{fig:REZ}
\end{figure}

\subsection{Main results}
\label{sub:main_result}
The first main result in this paper establishes the optimal regularity
of the minimizers.
\begin{theorem}[Optimal regularity]\label{thm:C^1/2} If $u$ is a minimizer of the
  functional $J$ in \eqref{eq:energy} over $\K_0$ in \eqref{eq:K0},
  then $u\in C^{1/2}_{\loc}(B_1^+\cup B_1')$ with 
\[
\|u\|_{C^{1/2}(B_{1/2}^+\cup B_{1/2}')}\leq C_n\|u\|_{L^{2}(B_{1}^+)}.
\]
\end{theorem}

The regularity above implies that for any $x\in\Gamma(u)$ we have
$$
\kappa=N^x(0+,u)\geq 1/2.
$$
The knowledge of the possible values of $\kappa$ is important for the
classification of free boundary points (as we discussed at the end of
subsection \ref{sub:the_problem}).  Concerning these values we have
the following results.

\begin{theorem}[Minimal Almgren's frequency at contact points]\label{thm:contact points}
  If $u$ is a minimizer of $J$ over $\K_0$, then for a contact point
  $\bar x\in \Gamma'(u)$ we have
		$$\kappa=N^{\bar x}(0+,u) \geq 3/2.$$
\end{theorem}
At non-contact points we give a more complete picture.
\begin{theorem}[Almgren's frequency at non-contact
  points]\label{thm:noncontact}
  If $u$ is a minimizer of $J$ over $\K_0$, then for a non-contact
  point $\bar x\in \Gamma^*(u)$ we have that
		$$\kappa=N^{\bar x}(0+,u)$$
                can take only the following values:
$$ 
1/2,\ 3/2,\ 5/2,\ \ldots,\ m-1/2,\ \ldots.
$$
\end{theorem}

\section{Optimal Regularity}
\label{sec:opt-reg}

\subsection{Symmetrization}
\label{sub:symmetrization}
It will be convenient for our considerations to extend every function
$v\in\K_0$ by even symmetry in $x_n$-variable to the entire ball
$B_1$:
$$
v(x',-x_n):=v(x',x_n)\quad\text{for }(x',x_n)\in B_1^+.
$$
With such extension in mind, the energy $J$ in \eqref{eq:energy} can
be replaced with
\begin{equation}
  J(v):=\frac12\int_{B_1} |\nabla v|^2. 
  \label{eq:energy_symm}
\end{equation}

\subsection{H\"{o}lder continuity}
\label{sub:h"_o_lder_regularity}
As the first result towards the optimal regularity, we show that the
minimizers are $C^\alpha$ regular for some $\alpha>0$.

\begin{proposition}[H\"older continuity]\label{prop:Holder} If $u$ is a minimizer of $J$ over $\K_0$, then
  $u\in C^\alpha(B_{1/2})$, with a dimensional constant $\alpha>0$ and
  \[
  \|u\|_{C^\alpha(B_{1/2})}\leq C_n \|u\|_{L^2(B_{1})}.
  \]
\end{proposition}

We start by showing that the positive and negative parts of the
minimizer $u$ are subharmonic. Note that at this stage we have not yet
established the continuity of $u$, so we will resort to the energy
methods.

\begin{lemma}\label{lem:subharmonicity} $u_{\pm}=\max\{\pm u , 0\}$
  are subharmonic functions in $B_1$.
\end{lemma}

\begin{proof}\label{pf:subharmonicity}
  Proving the lemma is equivalent to showing that for any nonnegative
  test function $\eta \in C_0^\infty({B_1})$ we have
  \begin{equation} \label{eq:subharm} \int_{B_1} \nabla u_{\pm} \nabla
    \eta \leq 0.
  \end{equation}
  Let $\psi_\epsilon\in C^\infty(\R)$ be a nondecreasing function such
  that
$$
\psi_\epsilon =0\quad\text{in }(-\infty,\epsilon),\quad
0\leq\psi_\epsilon\leq 1 \quad\text{in }(\epsilon,2\epsilon),\quad
\psi_\epsilon=1\quad\text{in }(2\epsilon,\infty).
$$
Then for a fixed $\epsilon >0$ and sufficiently small $|t|$ we have
\begin{align*}
  \{u>0\}= &\{ u +t \eta \psi_\epsilon(u_{\pm}) >0\} \\
  \{u<0\}= &\{ u+t \eta \psi_\epsilon(u_{\pm}) <0\}
\end{align*}
and thus $u+t\eta \psi_\epsilon(u_{\pm}) $ are admissible functions
from $\K_0$.  Since $u$ is a minimizer, we have $J(u+t\eta
\psi_\epsilon(u_{\pm}) )\geq J(u)$, yielding
\begin{align*}
  0=\int_{B_1} \nabla u \nabla \left(\eta \psi_\epsilon(u_{\pm})
  \right ) = \int_{B_1} \nabla u \nabla \eta \psi_\epsilon(u_{\pm})
  \pm \int_{B_1}|\nabla u|^2 \psi'_\epsilon(u_{\pm})\eta.
\end{align*}
Since the second integral is nonnegative, sending $\epsilon $ to $0$
we obtain \eqref{eq:subharm}.
\end{proof}

Once we know that $u_\pm$ are subharmonic in $B_1$, we immediately
obtain that $u$ is locally bounded.

\begin{lemma}[Local boundedness]\label{lem:loc-bound}
  If $u$ is a minimizer of $J$ over $\K_0$, then $u\in
  L^\infty(B_{3/4})$ and more precisely
  \[
  \pushQED{\qed} \sup_{B_{3/4}}{|u|} \leq C_n \| u \|_{L^2(B_1)}.
  \popQED
  \]
\end{lemma}

We can now proceed to the proof of H\"older continuity.

\begin{proof}[Proof of Proposition~\ref{prop:Holder}.]
  Using the local boundedness and the fact that $u_\pm$ vanish on
  $B_1'\cap\{x_1\leq 0\}$, by the comparison principle we can write
  that
  \begin{equation}\label{eq:uMv}
    |u|\leq Mv\quad\text{in }B_{3/4},
  \end{equation}
  where $M=C_n\|u\|_{L^2(B_1)}$ and $v$ solves
  \begin{equation}\label{eq:v}
    \begin{aligned}
      \Delta v=0&\quad\text{in } B_{3/4}\setminus
      (B'_{3/4}\cap\{x_1\leq
      0\})\\
      v=0&\quad\text{on }B'_{5/8}\cap\{x_1\leq
      0\}\\
      v=1&\quad\text{on }\partial B_{3/4}
    \end{aligned}
  \end{equation}
  with boundary values changing continuously from $0$ to $1$ in
  $(B'_{3/4}\setminus B'_{5/8})\cap \left \{x_1 \leq 0\right \}$.  We
  next claim that the barrier function $v$ above is in
  $C^\alpha(B_{1/2})$.  Indeed, we can use a bi-Lipschitz
  transformation to map $B_{3/4}\setminus (B'_{3/4}\cap\{x_1\leq 0\})$
  to $B_{3/4}^+$ preserving the distance from the origin. Then $v$
  will transform into $w$, which would be a solution of a uniformly
  elliptic equation in divergence form with measurable coefficients:
  \begin{equation} \label{eq:div-form}
    \begin{aligned}
      \div (a_{ij}w_j)=0  &\quad\text{in } B_{3/4}^+\\
      w=0 &\quad \text{on } B'_{5/8}.
    \end{aligned}
  \end{equation}
  By the De Giorgi-Nash theorem, we know $w \in C^\alpha(B^+_{1/2})$,
  and since the transformation is bi-Lipschitz we also get $v\in
  C^\alpha(B_{1/2})$, which provides
  \begin{equation} \label{eq:mapped_u} |v(x)| \leq C \dist (x,
    B'_1\cap \left \{x_1 \leq 0\right \})^\alpha.
  \end{equation}
  The latter, together with \eqref{eq:uMv} gives
  \begin{equation} \label{eq:mapped_u1} |u(x)| \leq C M \dist (x,
    B'_1\cap \left \{x_1 \leq 0\right \})^\alpha.
  \end{equation}
  Combined with the next lemma, this implies $u\in C^\alpha(B_{1/2})$.
\end{proof}
\begin{lemma}\label{lem:holder-reg} Let $u$ be a minimizer of $J$ over
  $\K_0$. If for a $0< \beta \leq 1$ and all $x,y\in B_{1/2}$ the
  following property holds:
  \begin{equation} \label{eq:dist} |u(x)| \leq C_0 \dist (x, B'_1\cap
    \left \{x_1 \leq 0\right \})^\beta
  \end{equation}
  then $u\in C^\beta (B_{1/2})$ with $\|u\|_{C^\beta (B_{1/2})}$
  depending only on $C_0$, $n$, $\beta$.
\end{lemma}
\begin{proof}\label{pf:holder-reg}
  Denote $d_x:=\dist (x, B'_1\cap \left \{x_1 \leq 0\right \})$. Take
  any $x,y \in B_{1/2}$. Without loss of generality we can assume
  $x\in B^+_1$ and $d_y\leq d_x$. We will consider three cases:
  \begin{enumerate}[1)]
  \item $|x-y|>d_x/8$. Using \eqref{eq:dist} we get
  $$
  |u(x)-u(y) | \leq C_0 (d_x^\beta + d_y^\beta)\leq 2 C_0 8^\beta
  |x-y|^\beta.
  $$
\item $|x-y|\leq d_x/8$ and the $n$-th coordinate of $x$, $x_n>
  d_x/4$.  In this case we observe that $B_{d_x/4}(x)\subset
  B^+_1$ and thus $u$ is harmonic there, $x,y\in B_{d_x/8}(x)$ and the
  interior gradient estimates for harmonic functions imply
  \begin{align*}
    |u(x) - u(y)| &\leq C_n \|u\|_{L^{\infty}(B_{d_x/4}(x))} \frac{
      |x-y| }{ d_x } \\ &\leq C_n C_0 (5/4)^\beta d_x^\beta \frac{
      |x-y|^\beta (d_x/8)^{1-\beta } }{ d_x }= C|x-y|^\beta.
  \end{align*}
  
\item $|x-y|\leq d_x/8$ and $x_n \leq d_x/4$.  In this case
  $B_{(3/4)d_x}(x',0)\subset B_{d_x}(x)$. Thus $u$ solves the Signorini
  problem in $B_{(3/4)d_x}(x',0)$ and $x,y\in B_{(3/8)d_x}(x',0)$. Using
  the interior Lipschitz regularity for the solutions of the Signorini
  problem, see \cite[Theorem 1]{AC}, we have
  $$
  |u(x) - u(y)| \leq C_n \|u\|_{L^{\infty}(B_{(3/4)d_x}(x))} \frac{
    |x-y| }{ d_x }
  $$
  and we complete the proof as in the previous case.\qedhere
\end{enumerate}
\end{proof}

\subsection{Monotonicity formula in the halfball}
\label{sub:half_monotonicity_formula}
As we observed in the introduction, we know that the function $\hat
u_{1/2}$ restricts the regularity of our solutions to $C^{1/2}$. In
order to rigorously obtain that $C^{1/2}$ is also the minimum expected
(and thus optimal) regularity, we need the following monotonicity
formula for the halfball, first introduced in \cite{AC}.
\begin{lemma}[{Monotonicity formula, \cite[Lemma
    4]{AC}}]\label{lem:half-monotonicity}
  For any $w\in C(\overline{B_1^+})$ satisfying
  \begin{alignat*}{2}
    \Delta w  &=0     \quad \text{ in } B_1^+,\\
    w &= 0 \quad \text{ on } B_1' \cap  \left \{x_1 \leq 0\ \right \},\\
    w\geq 0,\quad w\partial_{x_n} w &= 0 \quad \text{ on } B_1'.
  \end{alignat*}
  Then the function
  $$
  \phi(r):=\frac{1}{r} \int_{B_r^+} \frac{ |\nabla w|^2 }{ |x|^{n-2} }
  \d x
  $$
  is nondecreasing for $r \in (0, 1)$.
\end{lemma}
\begin{proof}\label{pf:half-monotonicity}
  The proof is a verbatim repetition of that of \cite[Lemma 4]{AC},
  despite of the slight difference in the assumptions. Namely, instead
  of asking the convexity of the set $\{x'\in B_1' :w(x',0) > 0\}$, we
  note that it is only used to show that the complement set of the
  support of $w$ contains the lower dimensional halfball $B_1'\cap \{
  x_1 \leq 0\}$, which is automatically satisfied in the setting of
  our problem.
\end{proof}

\subsection{Optimal $C^{1/2}$ regularity of minimizers}
\label{sub:optimal_c_alpha_b__1_2_}
We are now ready to proof our first main result.
\begin{proof}[Proof of Theorem \ref{thm:C^1/2}.]
  We apply the monotonicity formula in Lemma
  \ref{lem:half-monotonicity} to the minimizer $u$ of $J$ to obtain
  \begin{equation} \label{eq:fbound} \phi(r) \leq \phi( 3/4 ) \leq C
    \| u \|_{L^2(B_1)}^2.
  \end{equation}
  Here, the last inequality is standard for non-negative subharmonic
  functions (for a proof see for example \cite{Ca2}). Applying this
  for $u_{\pm}$ we obtain the corresponding inequality for $u$.

  Now using the fact that $u$ vanishes on $ B_1' \cap \left \{x_1 \leq
    0\ \right \}$ we also have the Poincar\'e inequality for the
  halfball

  \begin{equation} \label{eq:fbound1} \int_{B_r^+} u^2 \leq C_n
    r^2\int_{B_r^+} |\nabla u|^2.
  \end{equation}
  Then by the scaling of Lemma~\ref{lem:loc-bound} we have
  \begin{equation} \label{eq:sup_u}
    \begin{aligned}
      \sup_{B_{r/2}}{|u|} &\leq C_n r^{-\frac{n}{2}} \| u\|_{L^2(B_r)}\leq C_n r^{1-\frac{n}{2}}  \| \nabla u \|_{L^2(B_r^+)}\\
      & \leq C_n r^{\frac{1}2} \left(\frac 1 r \int_{B_r^+}
        \frac{|\nabla u|^2}{|x|^{n-2}}\d x\right)^{1/2} \leq C_n
      r^{\frac{1}2} \phi(r)^{1/2} \leq C_n r^{\frac{1}2} \| u
      \|_{L^2(B_1)}.
    \end{aligned}
  \end{equation}
  Let us notice that the above estimate holds also for any ball
  $B_{r/2}(x)$ with a center $x\in B'_{1/2} \cap \{x_1 \leq 0\}$, and
  $r\leq 1/4$
  \begin{equation} \label{eq:distLinfty} \sup_{B_{r/2}(x)}{|u|} \leq
    C_n r^{\frac{1}2} \| u \|_{L^2(B_1)} \leq C_n r^{\frac{1}2} \| u
    \|_{L^\infty (B_1)}
  \end{equation}
  yielding
  \begin{equation}
    |u|  \leq C \dist (x, B'_1\cap \{ x_1 \leq 0 \} )^{1/2}. 
  \end{equation}
  Using Lemma \ref{lem:holder-reg} we obtain $u \in C^{1/2}(B_{1/4})$.
\end{proof}

\begin{remark}
  Without loss of generality we will further assume that $u\in
  C^{1/2}(B_1)$.
\end{remark}

\section{Monotonicity of the Frequency}
\label{sec:almgr-freq-form}
\subsection{Almgren's Frequency Formula}
\label{sub:amgren_s_frequency_formula}

As we mentioned in the introduction, Almgren's frequency formula plays
and important role in the Signorini problem. Since we have an
additional constraint for functions in $\K_0$, it is not automatic that it
will still be monotone. Fortunately, however, it is still the case.

\begin{theorem}[Monotonicity of the frequency]\label{thm:almgren}
  If $u$ is a minimizer of $J$ over $\K_0$, then
  \begin{equation} \label{eq:almgren} N(r)=N^{x_0}(r,u):= \frac{ r
      \int_{B_r(x_0)} |\nabla u |^2}{ \int_{\partial B_r(x_0)} u ^2}
  \end{equation}
  is monotone in $r$ for $r\in(0,R)$ and $x_0\in B_1'\cap \{ x_1 \geq
  0 \}$ such that $B_R(x_0)\subset B_1$. Moreover, $N^{x_0}(r,u)\equiv
  \kappa$ for all $0<r\leq R$ iff $u$ is homogeneous of degree
  $\kappa$ in $B_R(x_0)$, with respect to the center $x_0$.
\end{theorem}

The following notations will be used in the proof:
$$
D(r):=\int_{B_r(x_0)}|\nabla u |^2 \quad \text{and} \quad H(r):=
\int_{\partial B_r(x_0)} u ^2.
$$ 
Now if we consider the logarithm of $N(r)$ and formally differentiate
it, we obtain
$$
\frac{ N'(r) }{ N(r) }=(\log N(r))'= \frac{1}{r} + \frac{ D'(r) }{
  D(r) } - \frac{ H'(r) }{ H(r) }.
$$
In order to prove the theorem, we need to show that the right hand side is
nonnegative. We accomplish this by proving differentiation
formulas/inequalities in Lemmas~\ref{lem:first-id},
\ref{lem:hprime}, and \ref{lem:Rellich}, following similar proofs in \cite{GP} or \cite{ACS}.

\smallskip
We start with the following alternate formula for $D(r)$.

\begin{lemma}[First identity]\label{lem:first-id}
  For the minimizers $u$ of $J$ over $\K_0$, the following identity
  holds  for $B_r(x_0)\Subset B_1$ with $x_0\in B_1'$:
  \begin{equation} \label{eq:uunu} D(r)=\int_ {B_r(x_0)} |\nabla u|^2=
    \int_{\partial B_r(x_0)} u u_\nu.
  \end{equation}	
\end{lemma}

\begin{proof}\label{pf:first}
  To prove the lemma we note that for any test function $\eta \in
  W^{1,2}(B_r(x_0))$ which vanishes in a neighborhood of $B_1' \cap \{
  x_1 \leq 0 \}$ then we have
  \begin{equation} \label{eq:approx*} \int_{B_r^+(x_0)} \nabla u
    \nabla \eta = \int_ {B_r'(x_0)} u_\nu \eta + \int_ {(\partial
      B_r(x_0))^+} u_\nu \eta.
  \end{equation}	
  For a small $\epsilon>0$, choose $\eta_\epsilon(x) =u \psi (
  d(x)/\epsilon)$, where $d(x)=\dist\left(x, B_1' \cap \left\{ x_1
      \leq 0 \right \} \right)$ and $\psi\in C^\infty(\R)$ is such
  that
  \begin{gather*}
    \psi=0\quad\text{in }(-\infty,1),\quad 0\leq\psi\leq1\quad\text{in
    }(1,2),\quad \psi=1\quad\text{in }(2,\infty),\\ 0\leq \psi'\leq
    M\quad\text{in }(-\infty,\infty).
  \end{gather*}
  We want to plug $\eta=\eta_\epsilon$ into \eqref{eq:approx*} and let
  $\epsilon\to 0$. We first claim that
  \begin{equation} \label{eq:approx} \lim_{\epsilon \to 0
    }\int_{B_r^+(x_0)} \nabla u \nabla \eta_\epsilon =
    \int_{B_r^+(x_0)} |\nabla u |^2,
  \end{equation}
  which is the same as
  \begin{equation} \label{eq:approx1} \lim_{\epsilon \to
      0}\int_{B_r^+(x_0)} \nabla u (\nabla \eta_\epsilon - \nabla u)
    =0.
  \end{equation}
  Indeed
  \begin{align*}
    \nabla \eta_\epsilon &=  \psi\left(\frac{ d }{ \epsilon  }\right) \nabla  u + u \psi'\left(\frac{ d }{ \epsilon  }\right)\frac{ \nabla  d}{ \epsilon  },\\
    \nabla \eta_\epsilon - \nabla u &= \left( \psi\left(\frac{ d }{
          \epsilon }\right) -1 \right) \nabla u +\frac{u}{ \epsilon }
    \psi'\left(\frac{ d }{ \epsilon }\right)\nabla d.
  \end{align*}
  Multiplying both sides of the above by $\nabla u$ and integrating
  over $B_1$, we obtain
  \begin{equation} \label{eq:approx2} \left|\int_{B_r^+(x_0)}
      \nabla u (\nabla \eta_\epsilon - \nabla u)\right| \leq \int_{
      \{d\leq 2\epsilon \} } |\nabla u|^2 + \frac{ M }{ \epsilon }
    \int_{ \{d\leq 2\epsilon \} } u|\nabla u|,
  \end{equation}
  using that $|\psi'|\leq M$ and $|\nabla d|\leq1$.  Since the first
  integral on the right hand side goes to $0$ as $\epsilon \to 0$, it
  remains only to estimate the second one. We have
  \begin{equation} \label{eq:approx3} \frac{ M }{ \epsilon } \int_{
      \{d\leq 2\epsilon \} } u|\nabla u| \leq \left( \int_{ \{d\leq
        2\epsilon \} } |\nabla u|^2 \right)^{1/2} \frac{ M }{ \epsilon
    } \left( \int_{ \{d\leq 2\epsilon \} } u^2 \right)^{1/2}.
  \end{equation}
  Again the first integral goes to $0$, and to estimate the second one
  we use the $C^{1/2}$ regularity of $u$ to obtain
	$$u^2 \leq C\epsilon\quad  \text{in } \{d\leq 2\epsilon \}.$$ 
        Besides, we also have that $|\{d\leq 2\epsilon \} | \leq C
        \epsilon $, which gives
$$
\frac{1}{\epsilon}\left(\int_{\{d\leq 2\epsilon\}}u^2 \right)^{1/2}
\leq C
$$
and establishes \eqref{eq:approx}. Now, to complete the proof of the
lemma, we let $\eta=\eta_\epsilon$ in \eqref{eq:approx*} and pass to
the limit as $\epsilon\to 0$.  Using the fact that
$$
u_\nu \eta_\epsilon = u_\nu u \psi =0 \quad \text{on } B_1'
$$
we obtain \eqref{eq:uunu}.
\end{proof}

\begin{lemma}[Second identity]\label{lem:hprime}
  For the minimizer $u$ of $J$ over $\K_0$ the following identity
  holds for $B_r(x_0)\Subset B_1$ with $x_0\in B_1'$:
  \begin{equation} \label{eq:Hpr} H'(r) = \frac{ n-1 }{ r }H(r) + 2
    \int_{\partial B_r(x_0)} u u_\nu.
  \end{equation}	
\end{lemma}
The differentiation formula should be understood in the sense that
$H(r)$ is an absolutely continuous function of $r$ and that the
differentiation formula holds for a.e. $r$.

\begin{proof}\label{pf:hprime}
  We have
  \begin{align*}
    H(r)&= 2\int_{(\partial B_r(x_0))^+} u^2 = 2\int_{(\partial
      B_r(x_0))^+} \left(\frac{ x-x_0 }{ r } \nu \ u^2 \right)\\&
    = \frac{ 2 }{ r } \int_{B_r^+(x_0)}\div((x-x_0) u^2)\\
    &= \frac{ 1 }{ r } \int_{B_r(x_0)}\div(x-x_0) u^2 + \frac{ 2 }{ r
    } \int_{B_r(x_0)}(x-x_0)(\nabla u) u.
  \end{align*}
  Hence, we obtain
$$
H'(r)= \frac{ n }{ r } \int_{\partial B_r(x_0)} u^2 + \frac{ 2 }{ r }
\int_{\partial B_r(x_0)} (x-x_0)(\nabla u)u - \frac{ 1 }{ r } H(r),
$$
which yields the desired identity.
\end{proof}

While the above two identifies were the same as in the Signorini
problem, the third one becomes actually an inequality, which suffices
for our purposes. 
\begin{lemma}[Third (Rellich-type) inequality]\label{lem:Rellich}
  For the minimizer $u$ of $J$ over $\K_0$ the following inequality
  holds for $B_r(x_0)\Subset B_1$ with $x_0\in B_1'\cap\{x_1\geq 0\}$:
  \begin{equation} \label{eq:Dprime} D'(r) \geq \frac{ n-2 }{ r } D(r)
    + 2 \int_{\partial B_r(x_0)} u_\nu ^2
  \end{equation}
  or, equivalently,
  \begin{equation} \label{eq:rayleigh} r\int_{\partial
      B_r(x_0)}|\nabla u |^2 \geq \int_{B_r(x_0)}(n-2) |\nabla u |^2 +
    2r \int_{\partial B_r(x_0)} u_\nu ^2.
  \end{equation}
\end{lemma}
We explicitly observe that the center $x_0$ of the ball
$B_r(x_0)$ must be in the upper thin halfball $B_1'\cap \{x_1\geq 0\}$
for the inequality to hold.
\begin{proof}\label{pf:Relich}
  The proof of this lemma uses the domain variation in radial
  direction similar to the one in \cite[p.~444]{We}. The main
  difference is that our constraints allow us to make perturbations
  that increase the distance from the origin, thus yielding an
  inequality (with the correct sign) instead of the equality in the
  non-constrained case.  We consider the function
		$$
                \eta_k(y):=\max \Big\{ 0, \min \Big\{ 1,\ \frac{ r -
                  |y|}{ k } \Big \} \Big \}.
		$$
                Then for $\epsilon> 0$, we have
$$
u_\epsilon(x)=u(x+\epsilon \eta_k(x-x_0)(x-x_0)) \in \K_0.
$$
Note that the same will not be true for negative $\epsilon$ (which is
why we only have an inequality), that variation will bring over the
zero values of $u$ from $B_1'\cup \{ x_1\leq 0 \}$ into $B_1'\cup \{
x_1 > 0 \}$, rendering the variation not an admissible function. Once
we established the admissibility of $u_\epsilon$, we can translate
$x_0$ into the origin and continue the rest of the proof for balls
centered at the origin.

Using the minimality of $u$, we have
$$
0\geq \frac{J(u)-J(u_\epsilon)}{\epsilon}=\frac{J(u(x))- J(u(
  x+\epsilon \eta_k(x)x))}{ \epsilon }.
$$
Letting $\epsilon\to 0$ this gives
\begin{alignat*}{2}
  0 & \geq \int_{B_r} \left(   |\nabla  u|^2 \div(\eta_k(x)x)- 2\nabla  u D(\eta_k(x)x)\nabla  u \right) \\
  & = \int_{B_r} \left((n-2) |\nabla u|^2 \eta_k(x) + |\nabla u|^2 x
    \nabla \eta_k(x) - 2 (x \nabla u ) (\nabla u \nabla \eta_k(x))
  \right).
\end{alignat*}
Sending this time $k\to\infty$, we obtain
\begin{alignat*}{2}
  0 & \geq \int_{B_r} (n-2) |\nabla u|^2 - \int_{\partial B_r} \left(
    |\nabla u|^2 x \nu + 2 (x\nabla u ) (\nu \nabla u ) \right),
\end{alignat*}
which is equivalent to \eqref{eq:rayleigh}.
\end{proof}

We can now prove the monotonicity of Almgren's frequency.

\begin{proof}[Proof of Theorem~\ref{thm:almgren}]\label{pf:almgren}
  The three lemmas proved above imply
$$
\frac{ N'(r) }{ N(r) }\geq \frac{1}{r} + \frac{n-2}{r} - \frac{n-1}{r}
+ 2\left( \frac{ \int_{\partial B_r(x)} u_\nu ^2}{ \int_{\partial
      B_r(x)} u u_\nu} - \frac{ \int_{\partial B_r(x)} u u_\nu }{
    \int_{\partial B_r(x)} u ^2 } \right)\geq 0.
$$
The last inequality follows from the Cauchy-Schwartz inequality, the
equality case of which has to be satisfied if $N'(r)=0$ and provides
that $u$ is homogeneous (see \cite{GP} or \cite{ACS}). From the
scaling properties of $N(r,u)$ we can also see that it is constant
when the function $u$ is homogeneous, thus the theorem is proved.
\end{proof}

\section{Blowups and Possible Homogeneities}

\subsection{Blowups}
\label{sub:blow_ups}
An important tool for us will be the following rescaling of the
minimizers at some points $x_0\in \Gamma(u)$:
\begin{equation} \label{eq:rescaling} u_r(x)= u_{x_0,r}(x):= \frac{
    u(rx+x_0) }{ \left(\frac{ 1 }{ r^{n-1} } \int_{\partial B_r(x_0)}
      u^2 \right)^{1/2} }.
\end{equation}
The limits of the rescaled functions $\{ u_r\} $ as $r=r_j\to 0+$ will
be called \emph{blowups} of $u$ at the point $x_0$. The above
definition normalizes the $L^2(\partial B_1)$ norm of the rescaled
functions to be one:
\begin{equation} \label{eq:normu_r} \int_{\partial B_1} u_r^2=1.
\end{equation}
Another useful property is the following identity
\begin{equation} \label{eq:Nu_r} N(\rho , u_r) = N^{x_0}(\rho r, u).
\end{equation}
We next want to let $r=r_j\to 0$ and study the convergence of the
rescaled functions $u_{r_j}$. We start by showing that such
convergence will be strong in $W^{1,2}$.

\begin{lemma}[Strong convergence]\label{lem:convergence of rescalings}
  Let $u_j$ be a minimizer of $J$ over $\K_0(g_j)$ with some $g_j\in
  L^2((\partial B_1)^+)$.  Let also $\| u_j\|_{W^{1,2}(B_1)} \leq C $,
  and $u_j \rightharpoonup u_0 $ weakly in $W^{1,2}(B_1)$ and $u_j\to
  u_0$ in $C^\alpha_\loc (B_1) $. Then $u_j \to u_0 $ strongly in
  $W^{1,2}_{\loc}(B_1)$:
  \begin{equation} \label{eq:stong convergence} \int_{B_\rho} |\nabla
    u_j|^2 \to \int_{B_\rho} |\nabla u_0|^2 \quad \text{ for all }\ 0<
    \rho <1.
  \end{equation}
  Moreover $u_0$ minimizes $J$ over $\K_0(g_0)$ with boundary values
  $g_0=\lim_{j\to \infty}{g_j}$.
\end{lemma}
\begin{proof}\label{pf:lem-convergence of rescalings}

  1) We first prove that for any two solutions $u_1$ and $u_2$,
  $(u_2-u_1)_\pm$ are subharmonic:
  \begin{equation} \label{eq:diff-is-subhar} \int_{B_1} \nabla (u_2 -
    u_1)_\pm \nabla \eta \leq 0
  \end{equation}
  for all nonnegative test functions $\eta \in C_0^\infty(B_1)$.  We
  will show only the subharmonicity of $(u_2-u_1)_+$, the other one
  being analogous.  Now since the only complications can occur on
  $B_1'\cap \{ x_1 > 0 \} $, without loss of generality we may assume
  that $E=\{u_2>u_1\}\cap B_1'\subset B_1'\cap\{x_1>0\}$ is
  nonempty. Then from the Signorini conditions on $B_1'$ we have that
  \begin{equation*}
    \partial_{x_n} u_2 =0\quad\text{on }E,\quad
    \partial_{x_n} u_1 \leq 0\quad\text{on  }E,\quad
    \partial_{x_n} (u_2 - u_1) \geq 0\quad\text{on  }E.
  \end{equation*}
  For any point $x_0\in E$, let $\delta>0$ be such that $B'_\delta
  (x_0) \subset E$. Then from harmonicity of $u_2 -u_1$ in $B_1^\pm$,
  we have that for any test function $\eta \geq 0$, $\eta \in
  C_0^\infty(B_\delta (x_0))$,
$$
\int_{B_\delta(x_0)} \nabla (u_2 - u_1) \nabla \eta = 2
\int_{B_\delta^+(x_0)} \nabla (u_2 - u_1) \nabla \eta =-
\int_{B_\delta '(x_0)}\partial _{x_n} (u_2- u_1) \eta \leq 0.
$$ 
This implies the subharmonicity of $(u_2-u_1)_+$ in a neighborhood of
any point $x_0\in E$, implying the subharmonicity in $B_1$.

2) Take the sequence $\{ u_j \}$ and $u_0$ as in the statement of the
lemma. The previous step shows that $(u_j - u_k)_\pm$ are
subharmonic. Letting $k \to \infty$ we get $(u_j - u_0)_\pm$ is also
subharmonic. Now using the energy inequality we obtain
\begin{equation} \label{eq:uj-u0} \int_{B_\rho } |\nabla (u_j
  -u_0)_\pm |^2 \leq C(\rho ) \int_{B_1 } (u_j -u_0)_\pm ^2 \to 0
  \text{ as } j\to \infty,
\end{equation}
which implies the strong convergence in $B_\rho$.

3) Recall now that $u_j$ minimizes $J$ over $\K_0(g_j)$. Since $u_j$
are bounded in $W^{1,2}(B_1)$ and the trace mapping is compact, we can
take a subsequence such that $g_j\to{g_0}$ in $L^2(\partial B_1) $
as $j\to \infty$. Taking the minimizer $\hat u_0$ of $J$ on
$\K_0(g_0)$ and letting $u_0$ be the strong limit of $u_j$ obtained in
previous step and using that $(u_j -\hat u_0)_\pm$ is subharmonic, we
obtain
\begin{equation} \label{eq:uj-u0hat} \sup_{B_\rho}{|(u_j -\hat
    u_0)_\pm |} \leq C(\rho ) \int_{\partial B_1^+} (g_j -g_0)_\pm^2.
\end{equation}  
Thus $u_j$ converges uniformly to $\hat u_0$ on $B_\rho $ for any
$0<\rho <1$, meaning $\hat u_0 \equiv u_0$ in $B_1$.
\end{proof}

\subsection{Homogeneity of blowups}
\label{sub:homogeneity_of_blowups}

We next show that the blowups are homogeneous.

\begin{lemma}[Homogeneity of blowups]\label{lem:blowups are homog.} Let $u$ be a minimizer of
  $J$ over $\K_0$ and $u_0=\lim_{r_j \to 0}{u_{r_j}}$ to be a blowup
  of $u$ at $x_0\in \Gamma(u)$. Then $u_0$ is homogeneous of degree
  $\kappa=N^{x_0}(0+, u)$.
\end{lemma}
\begin{proof}\label{pf:blowups are homog.}
  Indeed, using \eqref{eq:Nu_r} and Theorem \ref{thm:almgren} for a
  $0< r <1/2$ we obtain
$$
N(1,u_r)=N(r,u)\leq N(1/2,u)=:M.
$$	
Using \eqref{eq:normu_r} and the above estimate we arrive at
$$
\int_{B_1} |\nabla u_r|^2 = N(1,u_r) \leq M,
$$
which shows that the sequence $\{ u_r \} $ is bounded in
$W^{1,2}(B_1)$. Thus, we can choose a weakly converging subsequence
$u_{r_j}\rightharpoonup u_0$ in $W^{1,2}(B_1)$. From Lemma
\ref{pf:lem-convergence of rescalings} we also have the strong
convergence $u_{r_j}\to u_0$ in $W^{1,2}_{\loc}(B_1)$, which means in
particular that
\begin{equation} \label{eq:limN(r_j)} \lim_{r_j\to 0} {N(\rho,
    u_{r_j})} = N(\rho, u_0),
\end{equation}
provided $ \int_{\partial B_\rho} u_0^2 \neq 0$. Now suppose
$\int_{\partial B_\rho} u_0^2=0$. Then by the maximum principle we
would have that the subharmonic functions $(u_0)_\pm$ vanish in
$B_\rho $, and since $u_0$ is harmonic in $B_1^+$, we obtain that $u_0
\equiv 0$ in $B_1$. But due to compactness of trace mapping, we have
$$
\int_{\partial B_1} u_0^2 = \lim_{r_j \to 0}{\int_{\partial B_1}
  u_{r_j}^2d\sigma =1},
$$
which contradicts to $u_0$ vanishing in $B_1$. Thus
\eqref{eq:limN(r_j)} holds for any $0< \rho <1$. Moreover we can write
$$
N(\rho, u_0) = \lim_{r_j\to 0} {N(\rho, u_{r_j})} = \lim_{r_j\to 0}
{N^{x_0}(\rho r_j, u)} = N^{x_0}(0+, u) =:\kappa,
$$
yielding
\begin{equation} \label{eq:N(u_0)=k} N(\rho, u_0) \equiv \kappa \quad
  \text{for any } 0< \rho <1.
\end{equation}
Then using the last part of Theorem~\ref{thm:almgren}, we complete the
proof of the lemma.
\end{proof}

We can now proceed to the proof of Theorems~\ref{thm:contact points} and
      \ref{thm:noncontact}.
      \subsection{Minimal homogeneity at contact points}
      \label{sub:contact_and_nonconcat_points}
      \begin{proof}[Proof of Theorem~\ref{thm:contact
          points}]\label{pf:Ncontact}
	For a fixed $r>0$ consider a functional
	\begin{equation} \label{eq:almgren2} \Gamma(u)\ni x\mapsto
          N^{x}(r,u)= \frac{ r \int_{B_r( x)} |\nabla u |^2}{
            \int_{\partial B_r(x)} u ^2 } .
	\end{equation}
	Then, since for a fixed  $r>0$ the functional above is
        continuous and that $N^x(r,u)$ is
        nondecreasing in $r$, we obtain the upper semicontinuity of
        the functional $x\mapsto N^{x}(0+,u)$ on $\Gamma(u)$.  More
        precisely, we have
	\begin{equation} \label{eq:upper-semic} N^{x_0}(0+,u)\geq
          \limsup_{\substack{x\to x_0\\x\in\Gamma(u)}}
          N^x(0+,u).\end{equation} Now, for a contact point $\bar x\in
        \Gamma'(u)$ we have a sequence of free boundary points $x_j
        \in \Gamma(u)\cap\{x_1>0\}$ converging to $\bar x$. Now, near
        $x_j$, the minimizer $u$ solves the Signorini problem and
        therefore we have
$$
N^{x_j}(0+,u)\geq 3/2,\quad\text{for }x_j\in \Gamma(u)\cap\{x_1>0\}
$$
and thus, using the upper semicontinuity, we conclude that
\[
N^{\bar x}(0+,u)\geq 3/2.\qedhere
\]
\end{proof}

\subsection{Possible homogeneities at non-contact points}
\label{sub:noncontact_points}
\begin{proof}[Proof of
  Theorem~\ref{thm:noncontact}]\label{pf:noncontact}
  Since $\bar x$ is not a contact point, we know that there exists a
  positive $\delta $ such that $u$ is harmonic in ${B_\delta} (\bar
  x)\setminus (B'_\delta(\bar x)\cap\{x_1\leq 0\})$. Let $u_0$ be a
  blowup of $u$ at $\bar x$:
$$
u_0=\lim_{r_j \to 0} u_{\bar x,r_j}= \lim_{r_j \to0}\frac{ u(\bar x+
  r_jx) } { \left(\frac{ 1 }{ r_j^{n-1} } \int_{\partial B_{r_j}(\bar
      x)} u^2 d\sigma\right)^{1/2} }.
$$ 
We know that $u_0$ is homogeneous of degree $\kappa=\kappa(\bar
x):=N^{\bar x}(0+,u)$, meaning $u_0(r \theta )=r^\kappa u_0(\theta )$
for $r>0$ and $\theta\in\partial B_1$. We also know that $u_0$ is
harmonic in $\R^n\setminus (\R^{n-1}\cap\{x_1\leq 0\})$ and $u_0$ is
nonnegative in $\R^{n-1} \cap\{ x_1 >0 \}$.
Next, for $m\in\N$, define
\begin{equation} \label{eq:u_n} \hat u_{m-1/2}(x):=
  \Re(x_1+i|x_n|)^{m-1/2}.
\end{equation}
It is easy to see that $\hat u_{m-1/2}$ is homogeneous
of degree $(m-1/2)$ and
\begin{align*}
  \Delta \hat u_{m-1/2} =0 &\quad\text{in }
  \R^n\setminus(\R^{n-1}\cap\{x_1\leq 0\}) \\
  \hat u_{m-1/2}=0 &\quad \text{on } \R^{n-1} \cap\{ x_1\leq 0 \}.
\end{align*}
Thus, the set of possible values of $\kappa$ includes $\{m-1/2:m\in\N\}$.
We want to show that those are the only possible values of $\kappa$. This
fact will follow from the expansion of harmonic functions in slit domains,
recently established in \cite[Theorem~3.1]{DS0}. The latter theorem
implies that for any $k\geq 0$ there exists a polynomial $P_0(x,r)$ of
degree $k+1$ such that
$$
u_0(x)=\hat u_{1/2}(x)\left(P_0(x',r)+o(|x|^{k+1})\right),\quad
r=\sqrt{x_1^2+x_n^2},
$$
solely from the fact that $u_0$ is harmonic in $B_1\setminus B_1'\cap\{x_1\leq 0\}$, vanishes continuously on $B_1'\cap\{x_1\leq 0\}$ and is even in $x_n$. 
Taking $k>\kappa$ and using that $u_0$ is homogeneous of degree
$\kappa$, we obtain that
$$
u_0(x)=\hat u_{1/2}(x)P_0(x',r)
$$
for a homogeneous polynomial $P_0(x',r)$ of degree $\kappa-1/2$. Thus,
$\kappa=m-1/2$ for some $m\in\N$. The proof is complete.
\end{proof}

\end{document}